\documentclass[11pt]{amsproc} 

\usepackage{amsmath, amsthm, amsfonts, amssymb}

\usepackage{amsrefs}
\usepackage{hyperref}
\usepackage{cleveref}
\usepackage[all]{xy}
\usepackage{mathrsfs}
\usepackage{enumerate}

\theoremstyle{plain}
\newtheorem{teo}{Theorem}[section]
\newtheorem{prop}[teo]{Proposition}
\newtheorem{cor}[teo]{Corollary}
\newtheorem{lema}[teo]{Lemma}
\theoremstyle{definition}
\newtheorem{defs}[teo]{Definition}
\newtheorem{exm}[teo]{Example}

\newtheorem{rmk}[teo]{Remark}
\newtheorem{notc}{Notation}

\def\C{\mathbb{C}}

\def\PP{\mathbb{P}}

\def\F{\mathcal F}

\def\O{\mathcal O}

\def\-#1{\overline{#1}}
\def\^#1{\widehat{#1}}
\def\~#1{\widetilde{#1}}
\def\sing{\mathrm{sing}}

\begin{document}

\title[Unfoldings and structural theorems]{Isotrivial unfoldings and structural theorems for foliations on Projective spaces.}
\author{Federico Quallbrunn}
\address{Departamento de Matem\'atica, FCEyN, Universidad de Buenos Aires, Ciudad Universitaria, Pabell\'on 1, Buenos Aires (Argentina)}
\email{fquallb@dm.uba.ar}
\begin{abstract}
Following T. Suwa, we study unfoldings of algebraic foliations and their relationship with families of foliations, making focus on those unfoldings related to trivial families.
The results obtained in the study of unfoldings are then applied to obtain information on the structure of foliations on projective spaces.
\end{abstract}•

\maketitle

\section{Introduction}


The objective of this work is twofold.
\smallskip

In a first stance, we want to investigate the relation between unfoldings and families of foliations in algebraic varieties (see definitions below), specially those unfoldings related with trivial families.
In this respect our principal result is \cref{correspondencia}, which can be viewed as a generalization of \cite[chapter 7]{suwa}.
\smallskip

In a second stance, we apply our results on unfoldings and deformations to study foliations on $\PP^n$ whose degree is low with respect to $n$.
In this respect our main results are \cref{teopapa} and \cref{coropapa}.

The result of \cref{teopapa} is obtained here by considering unfoldings of foliations on $\PP^2$.\medskip
 
The concept of unfoldings of foliations and its relations with deformation theory of foliations were first introduced and studied by Suwa in a series of papers (see \cite{suwa81, suwa83, suwa} and references therein).

An unfolding of a foliation $\F_0$ on a variety $X$ is a "bigger" foliation $\F$ on a variety $X\times S$ whose leaves contains those of $\F_0$ (see below for a precise definition).
So we can always think of a foliation of codimension $q$ on $\PP^{n}$ as being birationally equivalent to an unfolding of a foliation by curves on $\PP^{q+1}$.
This, in general, does not provide us with much information, but under certain hypotheses we can control what kind of unfolding we may
consider.

The special kind of unfolding that will give us structural information on the foliation will be Isotrivial unfoldings.
Unfoldings may be thought of as special kinds of families of foliations, a family on which not only the differential equations defining the foliation vary continuously but also the \emph{solutions} i.e.: the leaves vary continuously as well (this is, of course, very vague, again, see bellow for precise definitions).
Isotrivial unfoldings are those that are related to trivial families of foliations, with this vague intuition, they are families in which the equations stays still while the solutions vary.
They were studied by Suwa in the infinitesimal case, see \cite{suwa}.

Here we generalize Suwa's results on isotrivial unfolding to be able to deal with unfoldings and families parametrized by arbitrary schemes.
It turns out that isotrivial unfoldings have a rather rigid structure and so being able to identify a foliation on $\PP^n$ with an isotrivial unfolding says something about the structure of the foliation. 

\bigskip
In \Cref{isotru} we give the principal definitions and state the theorems on unfoldings that will serve us as tools to conclude statements about foliations on $\PP^n$.
The main result is \cref{correspondencia} which can be viewed as a first step in determining the representation of the functor that to every scheme $S$ associates the unfoldings of a given foliation parametrized by $S$, although such a line of investigation will not be pursued in this work.

\medskip
In \Cref{unfPn} we explain how foliations on $\PP^n$ may be considered as giving rise to unfoldings of foliations by curves, and when
this unfoldings may be taken isotrivial.

\medskip
In \Cref{gentrans} we deal with the technical issue of transversality, which is a condition we need to have to be able to apply the results of \Cref{isotru}.

\medskip
Finally, in \Cref{ufc} we apply the results of the previous sections to conclude some structural statements on foliations on $\PP^n(\C)$.
In particular, \cref{teopapa} follows as a particular case of a statement (\cref{casiahi}) valid for arbitrary dimensional foliations, although restrictions on the degree are always required.

\bigskip
The author was supported by a post-doctoral grant of CONICET.
The author is grateful to Jorge Vitorio Pereira for useful suggestions, and for corrections of crucial mistakes in earlier versions of this paper.
Also gratitude is due to Ariel Molinuevo for introducing the author to the concept of unfolding of a foliation,
 and to the Algebraic Geometry Seminar of Buenos Aires for its nurturing atmosphere and overall support.

\section{Isotrivial unfoldings}\label{isotru}

In order to treat infinitesimal unfoldings and its relations with deformations of foliations we will need to consider non-reduced schemes, and generalize the notion of a foliation to this setting. 
Luckily, a straightforward generalization will serve our purposes just fine.

\begin{defs}
Let $\mathcal{X}$ be a (separated) scheme of finite type over an algebraically closed field $k$.
An \emph{involutive distribution} over $\mathcal{X}$ is a subsheaf $T\F\subseteq T\mathcal{X}$ that is closed under the Lie bracket,
i.e.: for every open subscheme $U\subset \mathcal{X}$ we have
 \[[T\F(U), T\F(U)]\subseteq T\F(U).\]
The annihilator of an involutive distribution $I(\F) := \mathrm{ann}(T\F)\subseteq \Omega^1_\mathcal{X}$ is an \emph{integrable Pfaff system}.
An integrable pfaff system $I$ verifies the equation 
\[d(I(\F))\wedge \bigwedge^rI(\F)=0\subset \Omega^{r+2}_\mathcal{X};\]
where $d:\Omega^j_\mathcal{X}\to \Omega^{j+1}_\mathcal{X}$ is the de Rham differential, and $r$ is the generic rank of the sheaf $\Omega^1_\mathcal{X}/I$. 
The generic rank of $T\F$ will be called the \emph{dimension} of the foliation and noted $\dim \F$. 
\end{defs}

Although there is no form of the Frobenius integrability theorem valid for any scheme (possibly non-reduced) of finite type over a field, we will still refer to the data of an integrable Pfaff system or an involutive distribution as a \emph{foliation}. 
Note that with this definition foliations may have singularities.

\begin{defs}[Pull-back of foliations]
Given a morphism $p:\mathcal{X}\to \mathcal{Y}$ and a foliation on $\mathcal{Y}$ defined by a Pfaff system $I(\F)\subseteq \Omega^1_\mathcal{Y}$.
By considering $p^*:\Omega^1_\mathcal{Y}(U)\to \Omega^1_\mathcal{X}(p^{-1}U)$ the pull-back of $1$-forms
 we define a foliation $p^*\F$ on $\mathcal{X}$ by defining its Pfaff system to be 
the one generated by the pull-backs of forms in $I(\F)$.
We call this the \emph{pull-back foliation} of the foliation $\F$.
\end{defs}

\begin{defs}[Unfolding] Let $X$ be a non-singular algebraic variety over an algebraically closed field and let $\F_0$ be a foliation on $X$.
Let $S$ be any scheme (of finite type over the base field of X) and $s\in S$ a closed point.
We denote by $\pi_1$ and $\pi_2$ the projections of $X\times S$ to $X$ and $S$ respectively.
We denote by $D\pi_2$ the differential map
\[  D\pi_2|_{(x.s)} : T(X\times S)\otimes k((x,s))\to TS\otimes k(s).\]
An \emph{unfolding} of $\F_0$ parametrized by $(S,s)$ is a foliation $\F$ on $X\times S$ such that
	\begin{enumerate}
	\item The restriction of $\F$ to $X\times s$ is $\F_0$ i.e.: if we take the pull-back foliation $\iota^*(\F)$ as in the above definition  we get $\iota^*(\F)=\F_0$, here $\iota:X\times s \to X\times S$ is the inclusion.
	\item Dimensions of $\F$ and $\F_0$ are related as thus: $\dim \F=\dim \F_0 + \dim S$.

	\end{enumerate}•
\end{defs}
In the case where $X$ and $S$ are non-singular varieties over $\C$, the leaves of the foliation $\F_0$ are contained in the larger dimensional leaves of the unfolding $\F$. 
\par 
\medskip
Now we remind the definition of the \emph{relative tangent sheaf}.
Given a morphism ${f: \mathcal{X}\to \mathcal{Y}}$ the relative tangent sheaf $T_f\mathcal{X}$ is the dual of $\Omega^1_{\mathcal{X}|\mathcal{Y}}$.
Remember that $T_f\mathcal{F}$ is naturally a sub-sheaf of the tangent sheaf $T\mathcal{X}$, its sections are the vector fields on $T\mathcal{X}$ that are tangent to the fibers of $f$.
In the case where $f=\pi_2:X\times S\to S$ is the projection we note $T_{\pi_2}(X\times S)=T_S(X\times S)$ and similarly with the other projection.
\begin{rmk}\label{sum}
In the case of the product of $X$ and $S$ we have the decomposition of sheaves.
\[ T(X\times S)\cong T_S(X\times S)\oplus T_X(X\times S).\]
Moreover we have
\[T_S(X\times S)\cong \pi_1^*(TX),\qquad T_X(X\times S)\cong \pi_2^*(TS),\]
where $\pi_1$ and $\pi_2$ are the projections.
\end{rmk}

\begin{rmk}\label{inducedfamily}
Let $\F_0$ be a foliation over a variety $X$ and $\F$ an unfolding of $\F_0$ parametrized by $(S,s_0)$.
Then $\F$ induces a family of foliations over $X$ parametrized by $S$ (see \cite{paper1})  in the following way:
\par Let $I(\F)$ be the Pfaff system associated with $\F$.
Let $p:\Omega^1_{X\times S}\to \Omega^1_{X\times S|S}$ be the projection from the sheaf of differential to the sheaf of relative differentials.
We set $I_S(\F):=p(I(\F))\subseteq  \Omega^1_{X\times S|S}$.
Note that $I_S(\F)$ is a family of \emph{integrable Pfaff systems} in the sense of \cite{paper1} such that its restriction to $s_0$ is $I(\F_0)$.
\par We can calculate the annihilator $T_S \F$ of $I_S(\F)$ which will be, of course, a family of \emph{involutive distributions}.
Indeed we obtain $T_S(\F)$ as the intersection of the subsheaves $T\F$ and $T_S(X\times S)$ of $T(X\times S)$.
\par 
\smallskip In general, given a family of foliations, it is not possible to glue the leaves of the different foliations in the family to higher dimensional leaves. 
So not every family of foliation comes from an unfolding, as a matter of fact, that is almost never the case.
\end{rmk}
 \begin{defs}[Isotriviality]
We say the unfolding $\F$ of $\F_0$ is \emph{isotrivial} if it induces a trivial family of Pfaff systems (equivalently of distributions), i.e.: if $I_S(\F)=\pi_1^*I(\F_0)$ as subsheaves of $\Omega^1_{X\times S|S}$, where ${\pi_1:X\times S\to X}$ is the projection and the morphism $\pi_1^*:\Omega^1_{X}\otimes \O_{X\times S}\to \Omega^1_{X\times S}$ is the pull-back of forms.
\end{defs}

\begin{defs}[Transversality]\label{trans}
Given a foliation $\F_0$ and an unfolding $\F$ of $\F_0$ parametrized by $(S,s_0)$ we have  exact sequences.

\[\xymatrix{
	&0 \ar[r] &T_S\F \ar[r] \ar@{^(->}[d] & T_S(X\times S) \ar@{^(->}[d] \ar[r] &N_S\F \ar@{^(->}[d] \ar[r] &0\\
	&0 \ar[r] &T\F \ar[r]  & T(X\times S)  \ar[r] &N\F  \ar[r] &0.\\
}
\]
Where $T_S\F$ is  defined as in \cref{inducedfamily}.
We say $\F$ is \emph{transversal} to $S$ if $N\F/N_S \F =0$.
\end{defs}

\begin{exm}
Let $X=\mathbb{A}^n=\mathrm{Spec}(k[x_1,\dots,x_n])$ and $S=\mathbb{A}^1=\mathrm{Spec}(k[y])$ be affine spaces.
Let $\F_0$ be a foliation on $\mathbb{A}^n$ with tangent sheaf $T\F_0=k[x_1,\dots,x_n]\cdot(X_1,\dots, X_r)$, the $X_i$'s being vector fields on $\mathbb{A}^n$.
\par Lets see what an isotrivial and transversal unfolding $\F$ of $\F_0$ parametrized by $\mathbb{A}^1$ should look like.

\smallskip In the first place, being isotrivial, the subsheaf $T_{\mathbb{A}^1}\F\subseteq T\F$ will be generated by the vector fields $X_1,\dots, X_r$, viewed as vector fields on $\mathbb{A}^n\times \mathbb{A}^1$.
So we can write $T\F$ as generated by vector fields $X_1,\dots, X_r, Y_1,\dots, Y_s$.
Now as $\F$ is an unfolding of $\F_0$, then $Y_1,\dots, Y_r$ must span a space of dimension $\leq 1$ on the tangent space to each point.
So we have 
\[T\F=k[x_1,\dots,x_n,y]\cdot(X_1,\dots, X_r, Y).\]
Moreover, as $\F$ is transversal over $\mathbb{A}^1$, we can write $Y=\~{Y}+\frac{\partial}{\partial y}$, where 
\[
\~{Y}=f_1(x,y)\frac{\partial}{\partial x_1}+\dots +f_n(x,y)\frac{\partial}{\partial x_n}.
\]
Then the involutivity of $T\F$ implies that, for each fixed $y_0\in \mathbb{A}^1$ we have a vector field $\~{Y}(-,y_0)$ on $\mathbb{A}^n$ that verifies
\[ [T\F_0, \~{Y}(-,y_0)]\subseteq T\F_0 .\]
Note that every family $\~{Y}$ of vector fields on $\mathbb{A}^n$ satisfying the above condition give rise to an isotrivial transversal unfolding of $\F_0$.

\smallskip
Essentially, the same is true for any isotrivial transversal unfolding of a foliation on a non-singular variety $X$, although the role of the family of vector fields $\~{Y}$ will be taken by a section of a certain sheaf on the parameter space $S$.
The rest of the section is devoted to the generalization of the above example.
\end{exm}

\begin{defs}
Remember that, given an involutive distribution $T\F$, the  bracket of vector fields defines a map $[-,-]:{T\F\otimes N\F}\to N\F$ known as \emph{Bott connection}. 
Using the Bott connection we will define a subsheaf of $N_S\F$, which we call $\mathfrak{u}(\F)$, as the subsheaf of $N_S\F$ whose local sections on an open set $V\subset X\times S$ are the following:
\[
\mathfrak{u}(\F)(V):=_{\mathrm{Def}}\{s\in N_S\F(V) \ \text{s.t.:}\ [T_S\F, s] = 0 \}.
\]
Note that $\mathfrak{u}(\F)$ is \emph{not} a coherent subsheaf of $N_S\F$ but only a subsheaf of $k$-vector spaces.
\end{defs}
\begin{notc}
 We denote $\Upsilon(\F_0):=_{\mathrm{Def}} \Gamma(X, \mathfrak{u}(\F_0))$.
\end{notc}

\begin{rmk}
The sheaf $\mathfrak{u}(\F)$ inherits from $T_S(X\times S)$ the structure of a sheaf of Lie algebras. 
Indeed, if $W\subseteq T_S(X\times S)$ is the preimage of  $\mathfrak{u}(\F)$ by the projection  $T_S(X\times S)\to N_S(\F)$, then $T_S\F$ is an ideal inside of $W$.
In particular $\Upsilon(\F)$ is a Lie algebra over the field of definition $k$.
\end{rmk}

The sheaf $\mathfrak{u}(\F)$ will be useful to study the relation between \emph{infinitesimal} unfoldings (i.e.: those parametrized by the Spec of artinian algebras) and unfoldings parametrized by varieties such as $\PP^1_k$.
\bigskip

Let us begin by considering an unfolding $\F$ of a foliation $\F_0$ parametrized by $S$, that is a foliation over $X\times S$.   
By \cref{sum} we have a projection $T(X\times S)\to T_S(X\times S)$, from this we get a projection
\[ 
T(X\times S)/T_S\F \to T_S(X\times S)/T_S\F = N_S\F.
\]

We will focus now on the properties of the map $\upsilon_{\F}$ one gets by composing
\begin{equation}\label{map}
T\F/T_S\F\to T(X\times S)/T_S\F \to N_S\F.
\end{equation}•
Note that, when the unfolding is transversal, we have an isomorphism $\pi_2^* TS\cong T\F/T_S\F$, so we can consider $\upsilon_\F$ to be a morphism 
\[ \upsilon_\F: \pi_2^* TS \to N_S\F. \]

\begin{prop}\label{upsilon}
If $\F$ is transversal to $S$ then $\upsilon_{\F}(\pi_2^{-1}TS)$ is a subsheaf of $\mathfrak{u}(\F)\subseteq N_S\F$.
\end{prop}
\begin{proof}
Note that the statement is making reference to $\pi_2^{-1}TS\subset \pi_2^{*}TS$, that is the sheaf of vector fields that are constant along the fibers of $\pi_2$.
\par So, given a local section $s\in \upsilon_{\F}(\pi_2^{-1}TS)$ we need to compute $[T_S\F, s]$.
To do this we take a pre-image of $s$, say $\~{s}\in T\F$.
As $p_2:T(X\times S)\to \pi_2^*TS$ induces an isomorphism between $T\F/T_S\F$ and $ \pi_2^*TS$ and  $s\in \upsilon_{\F}(\pi_2^{-1}TS)$, we can take $\~{s}$ of the form $Y+Z$ with $Y\in \pi_1^*TX$ and $Z\in  \pi_1^{-1}TS$.
Given $W\in T_S\F$ we compute
\[ [W, \~{s}]= [W, Y+Z] = [W,Y]- Z(W),\] 
as $W(Z)=0$, being $Z$ in $\pi_1^{-1}TS$.
Then $[W, \~{s}]$ is in $\pi_1^*TX$, and also in $T\F$, so it is in $T_S\F$.
Therefore $[T_S\F,s]=0$.
\end{proof}

\begin{teo}\label{correspondencia}
Let $X$ be a non-singular variety and $\F_0$ a foliation on $X$.
\par There is, for each scheme $S$, a $1$ to $1$ correspondence:
\[
\left\{\begin{aligned}
\text{isotrivial transversal unfoldings } \\
\text{of $\F$ parametrized by $S$} 
\end{aligned}• \right\}\longleftrightarrow \left\{
\begin{aligned}
\text{sections } \upsilon \in H^0(S,\Omega^1_S)\otimes \Upsilon(\F_0)
\\ \text{ s.t.: } d\upsilon+\frac{1}{2} [\upsilon,\upsilon]=0
\end{aligned}
\right\}.
\]
\end{teo}
\begin{proof}
 \hfill
\begin{description}
\item[1.Form associated to an unfolding]
Given an isotrivial transversal unfolding $\F$ of $\F_0$ we associate to it the map
\[ T\F/T_S\F \to N_S\F, \]
 as in  \cref{map}.
As $\F$ is transversal we have $T\F/T_S\F \cong T_X(X\times S)\cong \pi_2^*(TS)$.
So we have a map
 \[ \upsilon_\F: \pi_2^*(TS)\to N_S\F. \]
By \cref{upsilon}, the image of $ \pi_2^{-1}(TS)$ under $\upsilon_\F$ is a subsheaf of $\mathfrak{u}(\F)$.
Now, $\F$ being isotrivial implies that $N_S\F\cong\pi^*_1(N\F_0)$ and that $\mathfrak{u}(\F)\cong\pi_1^*(\mathfrak{u}(\F_0))$.

\smallskip  Then we have a global section $\upsilon_\F \in H^0(\pi_2^*\Omega^1_S\otimes \pi_1^*(N\F_0))$, by K\"unneth isomorphism this is a section in $H^0(\Omega^1_S)\otimes H^0(N\F_0)$ and by \cref{upsilon} this section actually belongs to  $H^0(\Omega^1_S)\otimes H^0(\mathfrak{u}(\F_0))$. 
So we get a global form 
\[ \upsilon_\F\in H^0(S,\Omega^1_S)\otimes \Upsilon(\F_0) \]
 associated with an isotrivial transversal unfolding $\F$ parametrized by $S$.

Now to establish the first part of the correspondence we need the following:
\begin{lema}\label{MCunf}
Let $\F_0$ be a foliation on a nonsingular variety $X$, and $\F$ an isotrivial and transversal unfolding of $\F_0$ parametrized by a variety $S$. 
The global $1$-form $\upsilon_\F$ verifies the Maurer-Cartan equation:
\[ d \upsilon_\F + \frac{1}{2} [\upsilon_F, \upsilon_F]=0. \]
\end{lema}
\begin{proof}[Proof of \cref{MCunf}]
Set $\upsilon=\upsilon_F$,  now given two vector fields $Y$ and $Z$ defined on $S$ we use Cartan's formula for the differential
\[ 
d\upsilon(Y, Z)= Y(\upsilon(Z)) -Z  (\upsilon(Y)) - \upsilon ([Y, Z]).
\]
Now we want to calculate $\upsilon ([Y, Z])$, for this we look at the definition of $\upsilon_\F$.\newline
Given a local section $Y$ of $TS$ we look at it as a section of $\pi^*(TS)$.
As was noted above, transversality of $\F$ gives us an isomorphism $T\F/T_S \F\cong \pi_2^*(TS)$, so we can associate to $Y$ a corresponding global section of $T\F/T_S\F$.

\smallskip
We then apply the restriction of the morphism $p_1: T(X\times S)\to \pi_1^*(TX)$ of \cref{sum} to the above section obtaining by \cref{upsilon} a section of $\mathfrak{u}(\F)$.

\smallskip
So $\widetilde{Y}$ viewed as a section of $T(X\times S)$ is of the form $Y+Y'$, with $Y'$ in $\pi^*_1(\mathfrak{u}(\F))$. 
Then we have $Y'=\upsilon(Y)$.

\smallskip
To calculate  $\upsilon ([Y, Z])$ we first observe that 
\[ 
\~{[Y,Z]}= [\~{Y}, \~{Z}].
\]
Indeed, the isomorphism between  $T\F/T_S$ and $ \pi_2^*(TS)$ that we use to define $\~{Y}$ comes from the inclusion of $T\F$ in $T(X\times S)$, so it respects Lie brackets.
So now $\upsilon ([Y, Z])$ is simply $p_1([\~{Y}, \~{Z}])$.

\smallskip
We can write $\~{Y}= Y+Y'$ and similarly with $Z$ and, noting that $p_1([Y,Z])=0$, compute
\begin{eqnarray*}
p_1([\~{Y}, \~{Z}])&=& p_1([Y+Y', Z+Z'])\\
 &=& [Y', Z'] + p_1([Y, Z']) -  p_1([Z, Y']) .
\end{eqnarray*}•
Now, as $ p_1([Y, Z'])= Y(Z')$ we have
\[
\upsilon([Y,Z])= [\upsilon(Y), \upsilon(Z)] + Y(\upsilon(Z)) - Z( \upsilon(Y)),
\]
from which the Maurer-Cartan equation, and therefore the lemma, follows.
\end{proof}

\item[2.Unfolding associated with a form] Given $\upsilon\in H^0(S,\Omega^1_S)\otimes \Upsilon(\F_0)$ we have an associated morphism
\[
\upsilon: \pi_2^*(TS)\to \pi^*_1(N\F_0).
\]
So we consider the morphism 
\[\upsilon\oplus \mathrm{id}:\pi_2^*(TS)\to \pi^*_1(N\F_0)\oplus \pi_2^*(TS)\]
given by $\upsilon\oplus \mathrm{id} (s)= (\upsilon(s), s)$.
We also have the projection $\phi: T(X\times S)\to \pi^*_1(N\F_0)\oplus \pi_2^*(TS)$.
So we consider the diagram
\[ 
\pi_2^*(TS)\xrightarrow{\upsilon\oplus \mathrm{id}} \pi^*_1(N\F_0)\oplus \pi_2^*(TS)\xleftarrow{\phi}T(X\times S).
\]

We then take $T\F_\upsilon\subseteq T(X\times S)$ to be the sub-sheaf generated by $\phi^{-1}(\upsilon(\pi_2^*(TS)))$.

\smallskip The fact that the image of $\upsilon$ is within $\Upsilon(\F_0)$ and that $\upsilon$ satisfies the Maurer-Cartan equation implies that $T\F_\upsilon$ is involutive.

Indeed, let $\~{Y}$ and $\~{Z}$ be local sections of $T\F_\upsilon$, we need to check that $[\~{Y},\~{Z}]$ is also a section of $T\F_\upsilon$.
So we take  have section $Y$ and $Z$ in $\pi_2^*(TS)$ such that 
\begin{eqnarray*}
\phi(\~{Y}) & = & (\upsilon(Y),Y)\\
\phi(\~{Z}) & = & (\upsilon(Z),Z)).
\end{eqnarray*}•
Moreover, we may assume that $Z$ and $Y$ are local sections of  $\pi_2^{-1}(TS)$, as this latter sheaf generates $\pi_2^*(TS)$ and the general result will follow from the fact that the Lie bracket is a derivation on both of its inputs.
So, given that $Y,\ Z\in \pi_2^{-1}(TS)$ we have
\begin{eqnarray*}
\phi([\~{Y},\~{Z}]) & = & [\phi(\~{Y}),\phi(\~{Z})]=\\
 & = & [Y,Z] + Y(\upsilon(Z)) - Z(\upsilon(Y)) + [\upsilon(Y), \upsilon(Z)]=\\
 & = & [Y,Z] + \upsilon([Y, Z]).
\end{eqnarray*}•
So, $T\F_\upsilon$ is involutive.

Also, by construction, $\F_\upsilon$ induces a trivial family.
Hence $\F_\upsilon$ is an isotrivial, transversal unfolding of $\F_0$, associated to an $\upsilon\in  H^0(S,\Omega^1_S)\otimes \Upsilon(\F_0)$.

 \end{description}
\medskip
It follows routinely that the constructions of $\upsilon_F$ and of $\F_\upsilon$ are inverse to each other.
\end{proof}

We thus see that the space $\Upsilon(\F)$ will be an important ingredient in studying isotrivial unfoldings of a foliation $\F$. 
This space is acted upon by the group of automorphism of the foliation, that is the group
\[
\mathrm{Aut}(\F):= \{g\in \mathrm{Aut}(X) \text{ s.t.: } g_*(T\F)=T\F\}.
\]
The Lie algebra of this group may be naturally identified with the global sections of the sheaf $\mathfrak{aut}(\F)$ whose local sections are
\[
\mathfrak{aut}(\F)(V) :=\{ \theta \in TX(V) \text{ s.t.: } [T\F, \theta]\subseteq T\F \},
\]
so $Lie(\mathrm{Aut}(\F))= H^0(X, \mathfrak{aut}(\F))$.

\begin{rmk}\label{unfmc}
Note also that we have a short exact sequence of sheaves
\[
0\to T\F \to \mathfrak{aut}(\F) \to \mathfrak{u}(\F)\to 0.
\]

In the particular case when $H^0(X,T\F)=H^1(X,T\F)=0$, which will be important to us later, we have the equality
\begin{equation}
Lie(\mathrm{Aut}(\F))= H^0(X, \mathfrak{aut}(\F))=H^0(X,  \mathfrak{u}(\F))= \Upsilon(\F).
\end{equation}•

In particular, in the case where $X$ a complex variety and $H^0(X,T\F)=H^1(X,T\F)=0$,  there is an unfolding associated to the Maurer-Cartan form of the group $\mathrm{Aut}(\F)$, which in this case take values in $\Upsilon(\F)$. 
So there is, by \cref{correspondencia}, an unfolding associated to the Maurer-Cartan form.
\end{rmk}

In the situation where $X$ and $S$ be varieties over $\C$, $\F_0$ a foliation on $X$ such that $H^0(X,T\F_0)=H^1(X,T\F_0)=0$ and $\F$ a transversal isotrivial unfolding of $\F_0$ parametrized by $S$.
Denote $\pi:\~{S}\to S$ the universal covering of $S$, $\upsilon_{MC}$ the Maurer-Cartan form of $\mathrm{Aut}(\F_0)$.
Let $\F_{MC}$ the unfolding associated to the Maurer-Cartan form (cf.: \cref{unfmc}), and $\upsilon\in \Omega^1_{\~{S}}\otimes \Upsilon(\F_0)$ the pull-back of  $\upsilon_\F$ by the universal covering map.
A direct application of Darboux's existence theorem gives us: 

\begin{cor}\label{darboux}
With hypotheses as in the above paragraph, there is a morphism $f:\~{S}\to \mathrm{Aut}(\F_0)$ such that $\upsilon$ is the pull-back of $\upsilon_{MC}$.
Equivalently, $f^*(\F_{MC})=\pi^*(\F)$ as unfoldings of $\F_0$.
\end{cor}

\begin{exm}\label{exP2}
Let $\F_0$ be a foliation by curves on the complex projective plane $\PP^2(\C)$ such that  $H^0(X, \mathfrak{aut}(\F))\neq 0$.
Then $T\F_0$ is a line bundle on $\PP^2$, so $H^0(X,T\F_0)=H^1(X,T\F_0)=0$. 
Also, being a foliation by curves in $\PP^2$ implies that $\dim \mathrm{Aut}(\F_0) \leq 1$.
So there is an infinitesimal symmetry $Y$ such that $H^0(X, \mathfrak{aut}(\F))=(Y)$ and 
\[ \mathrm{Aut}(\F_0)^0=\exp(tY)\cong \C^*,\]
where $\mathrm{Aut}(\F_0)^0$ is the connected component of the identity.
In particular, we have that the universal covering $A$ of $\mathrm{Aut}(\F_0)^0$ is isomorphic as a complex variety to $\C$.
On $A$ we have the Maurer-Cartan form 
\[ \upsilon_{MC}=dz\otimes Y, \]
here we are taking $z$ to be a coordinate of $A\cong \C$ and we are identifying $\mathrm{Lie}(A)=\mathrm{Lie}(\mathrm{Aut}(\F_0))=(Y)$.

\smallskip
Considering that $\mathrm{Lie}(\mathrm{Aut}(\F_0))=\Upsilon(\F_0)$ and applying \cref{correspondencia}, this gives us the unfolding
$\F$ such that:
\[
T\F=(\pi_1^*T\F_0\oplus (\frac{\partial}{\partial z} + Y))
\]
on $\PP^2\times \C$.
If $\omega$ is the rational $1$-form on $\PP^2$ annihilating $T\F_0$, then $1$-form annihilating $T\F$ is 
\[ \varpi = \omega+ \omega(Y)dz\in \Omega^1_{\PP^2\times \C}. \]
so $\F$ posses the integrating factor $\omega(Y)$ (considered as a rational function on $\PP^2\times \C$).
Note that $\omega(Y)$ considered as a rational function on $\PP^2$ is an integrating factor for $\F_0$.
\end{exm}

\begin{rmk}
In a sense, the previous results generalize those of Suwa in \cite{suwa}.
There is proven, in the context of codimension $1$ foliations on $\PP^n$, a correspondence between infinitesimal isotrivial unfoldings (i.e.: isotrivial unfoldings parametrized by $k[x]/(x^2)$) and rational integrating factors of the foliation.
In our context this is understood the following way:
\par An infinitesimal isotrivial unfolding of a foliation $F$ will be transversal on some open set $\iota:U\hookrightarrow \PP^n$.
As we have seen, a transversal isotrivial unfolding of $\iota^*F$ give rise to a global section $s\in \Upsilon(\iota^*F)$. 
Restricting the open set $U$ further if needed we can take a section $Y$ in $Lie(\mathrm{Aut}(\iota^*F))$ representing $s$ modulo $T\iota^*F$.
In other words $Y$ is a \emph{rational symmetry} of $F$.
It is well known, see \cite{jvp}, that to every rational symmetry of a codimension $1$ foliation in a complete variety corresponds a rational integrating factor.
Thus we can recover Suwa's theorem from \cite{suwa}.
\end{rmk}


\section{Foliations on $\PP^n$ viewed as unfoldings}\label{unfPn}

In this section we will be interested in foliations on $\PP^n(\C)$ up to birational equivalence.
Moreover we will study the relations of foliations of arbitrary dimension on $\PP^n(\C)$ with foliations by curves on projective spaces of lower dimension.
First we recall an important definition.

\begin{defs}
A codimension $q$ foliation $\F$ on $\PP^n$ is said to be of \emph{degree} $d$ if and only if the associated integrable Pfaff system $I(\F)$ have the property that
\[
\wedge^q I(\F)\cong \O_{\PP^n}(-d-q-1). 
\]
\end{defs}

Given a foliation $\F$ on $\PP^n(\C)$ of codimension $q$, we fix a (rational, linear) projection $p:\PP^n\dashrightarrow \PP^{q+1}$.
Now, let $\mathrm{Gr}^{n}_{q+1}$ be the Grassmannian of $q+1$-dimensional linear spaces on $\PP^n$, define the open set
\[ 
U = \{ P\in \mathrm{Gr}^{n}_{q+1} \text{ s.t.: }  p|_P: P\to\PP^{q+1}\text{ is a (regular) isomorphism} \}\subset \mathrm{Gr}^{n}_{q+1} .
\]
Then we have for every $P\in U$ a foliation in $\PP^{q+1}$ given by first restricting $\F$ to $P$ and then applying the isomorphism $p$.
Note that if $\F$ is of degree $d$ so is every foliation on $\PP^{q+1}$ obtained this way.
\par 
In other words, what we are doing here is considering the incidence correspondence $Z=\{(x,P)\text{ s.t.: } x\in P\}\subset \PP^n\times \mathrm{Gr}^{n}_{q+1}$, intersecting with $\PP^n\times U$ gives us 
\[
Z\cap (\PP^n\times U) \cong \PP^{q+1}\times U
\]
and hence we have a diagram
\[
\xymatrix{
 & \PP^{q+1}\times U \ar[dl]_{\pi_1}\ar[dr]^{\pi_2} & \\
\PP^n&  &U.
}
\]
So taking the pull-back of $\F$ we have $\pi_1^*\F$ as a foliation on $ \PP^{q+1}\times U$.
Now, we can take the \emph{family of foliations over $U$} induced by $\pi_1^*\F$, as in \Cref{isotru}.
Restricting $U$ if necessary, we may assume that $\pi_1^*\F$ induces  a \emph{flat} family of involutive distributions over $\PP^{q+1}$, parametrized by $U$, lets call this family $F_U$.
We can characterize $F_U$ as follows, to each point $u\in U$ corresponds a ${q+1}-$linear space $P$ and $\pi_2$ is the projection from the incidence correspondence, so 
\[
F_U|_{\PP^{q+1}\times \{u\}}=p(\F|_P).
\]
By the results of \cite[Proposition 6.3]{paper1} such a family defines a morphism between $U$ and the moduli space of involutive distributions over $\PP^{q+1}$ of codimension $q$ and degree $d$.
This later space is  $\PP H^0(\PP^{q+1}, T\PP^{q+1}(d-1))$.

Then we have a morphism
\[
\phi_\F: U\to \PP H^0(\PP^{q+1}, T\PP^{q+1}(d-1)),
\]
the later being a projective space of dimension $(q+2)\binom{d+q+1}{d}-\binom{d+q}{d-1}$.
In particular if the dimension of $U$ (which is that of  $\mathrm{Gr}^{n}_{q+1}$) is greater than that of the target space, the morphism will have fibers of positive dimension.

\begin{rmk}\label{remfib}
Suppose the dimension of  $\mathrm{Gr}^{n}_{q+1}$ is greater than the dimension of $ \PP H^0(\PP^{q+1}, T\PP^{q+1}(d-1))$.
Set $V$ to be a fiber of $\phi_\F$, so $\phi_\F^{-1}(x)=V\subseteq U$, note that $V$ must have positive dimension. 
Then it follows from the universal property of the moduli space that the family $F_U|_V$ is trivial.
In particular the foliation $\pi_1^*\F|_{\PP^{q+1}\times V}$ defines an isotrivial unfolding parametrized by $V$.
\end{rmk}

\section{Generic transversality}\label{gentrans}

Now we investigate conditions under which an isotrivial unfolding turns out to be also transversal, so we can apply to it the theory of \Cref{isotru}.

\begin{defs}
Let $\F$ be a foliation on $X\times S$ viewed as an unfolding of foliations on $X$. We define $\mathscr{T}_\F$ to be the schematic support of the sheaf $N\F/N_S\F$.
\end{defs}

The subscheme $\mathscr{T}_\F$ will be of interest as it is the locus of points where transversality fails.

Recall that $\sing(\F)$, the singular locus of a foliation $\F$ on a scheme $\mathcal{X}$, is defined to be the schematic support of the sheaf $\mathcal{E}xt^1_\mathcal{X}(N\F,\O_{\mathcal{X}})$ (local Ext).
Similarly, if we have a family of involutive distributions parametrized by a scheme $S$ its singular locus is the scheme theoretic support of the sheaf $\mathcal{E}xt^1_\mathcal{X}(N_S\F,\O_{X\times S})$.

\begin{lema}
Let $X$ and $S$ be non-singular varieties.
Let $\F$ be an unfolding of foliations on $X$ parametrized by $S$. 
Suppose that $\F$ is an isotrivial unfolding of a foliation $\F_0$ on $X$, and that  $\mathscr{T}_\F \cap \mathrm{sing}(\F)=\emptyset$. 
Then $\mathscr{T}_\F\subseteq \sing(\F_0)\times S$.
\end{lema}
\begin{proof}
As $\F$ is an isotrivial unfolding $\sing(\F_0)\times S$ is the singular locus of the (trivial) family of involutive distributions induced by $\F$. So let $p\notin \sing(\F_0)\times S$, then  the localization $(N_S\F)_p$ is a free $\O_{X\times S}$-module and the short sequence
\[
0\to T_S\F\otimes k(p)\to T_S(X\times S)\otimes k(p)\to N_S\F\otimes k(p)\to 0
\]
is exact.
If $p\notin \sing(\F)$ then the sequence 
\[
0\to T\F\otimes k(p)\to T(X\times S)\otimes k(p)\to N\F\otimes k(p)\to 0
\]
is exact.
On the other hand we always have an immersion 
\[T_S(X\times S)\otimes k(p) \hookrightarrow T(X\times S)\otimes k(p).\]
Then if $p$ is a point neither in $\sing(\F_0)\times S$ nor in $\sing(\F)$ we have an immersion
\[ T_S\F\otimes k(p)\hookrightarrow T\F\otimes k(p). \]
So we have another immersion
\[ T\F/T_S\F\otimes k(p)\hookrightarrow T_X(X\times S).\]
As $\dim T\F=\dim T\F_0 + \dim S$ the dimension of the above vector spaces are equal to $\dim S$, so 
\[ N\F/N_S\F \otimes k(p) =0.\]
Then, if $p\notin \sing(\F_0)\times S$ and $p\notin \sing(\F)$, the point $p$ is not in $\mathscr{T}_\F$. 
\end{proof}

\begin{lema}
Let $X$ and $S$ be non-singular varieties over $\C$, denote $\pi_1$, $\pi_2$ the projections of $X\times S$ to the first and second factor, respectively.
Let $\F$ be an isotrivial unfolding of a foliation $\F_0$ on $X$ parametrized by $S$. 
Suppose that $\F_0$ is non-singular and that $\dim(\sing(\F))\leq \dim(\F)-1$.
Then $\pi_2|_{\sing(\F)}: \sing(\F)\to S$ is \emph{not} dominant.
\end{lema}
\begin{proof}
This follows from Theorem 2.7 of \cite{Suwasing}, actually we will use the following weaker version of the theorem in \cite{Suwasing}:
\par  If we take the reduced structure $\sing(\F)^{\text{red}}\subseteq \sing(\F)$
then, at a regular point $p$ of $\sing(\F)^{\text{red}}$, the image of the map
\[ T\F\otimes k(p)\to T(X\times S)\otimes k(p) \]
falls within the tangent space to $\sing(\F)^{\text{red}}$ at $p$.
\par To prove our assertion suppose that $\sing(\F)$ is dominant over $S$.
 Take $p$ to be a regular point of  $\sing(\F)^{\text{red}}$.
Then $\sing(\F)^{\text{red}}$ is dominant over $S$ as well, so it has dimension at least that of $S$.
On the other hand by Theorem 2.7 of \cite{Suwasing} the image of 
\[(T\F_0)\otimes k(\pi_1(p))\cong \pi_1^*(T\F_0)\otimes k(p)\cong T_S\F\otimes k(p)\to T(X\times S)\otimes k(p) \]
falls within the tangent space to $\sing(\F)^{\text{red}}$ at $p$.
And, as $\F_0$ is a non-singular foliation, we have that the dimension of that image is that of the leaves of $T\F_0$.
As $T_S\F$ is tangent to the fibers of $\pi_2$ the differential of the projection $D\pi_2$ satisfy
\[ D\pi_2(T_S\F)=0. \]

Then, comparing dimensions of tangent spaces we get
\[ \dim(\sing(\F))\geq \dim \F_0+\dim S = \dim \F, \]
obtaining a contradiction, so $\sing(\F)$ cannot be dominant over $S$.
\end{proof}

Putting together the last two lemmas we obtain the following.

\begin{prop}\label{proptrans}
Let $\F$ be an isotrivial unfolding  of a foliation $\F_0$  on a non-singular variety $X$ parametrized by a non-singular variety $S$, such that $\dim(\sing(\F))\leq {\dim(\F)-1}$. 
Set $Y=X\setminus \sing(\F_0)$.
Then there is an open set $U\subset S$ such that the restriction of $\F$ to $Y\times U$ is a \emph{transversal} isotrivial unfolding.
\end{prop}

\section{Unfoldings of Foliations by curves}\label{ufc}
Now we are in conditions of applying the results of \Cref{isotru} to foliations on $\PP^n$.
Let $\F$ be a foliation of degree $d$ in projective space $\PP^n$. 
If the condition 
\[ (n-q-1)(q+1)> q+2\binom{d+q+1}{d}-\binom{d+q}{d-1} \]
is satisfied, we are in the situation of \cref{remfib}.
Then there is an open set $U$ of $\mathrm{Gr}^{n}_{q+1}$ and a projection $p:\PP^n\dashrightarrow \PP^{q+1}$ trivializing the incidence correspondence, in such a way that we get a diagram
\[
\xymatrix{
 & \PP^{q+1}\times U \ar[dl]_{\pi_1}\ar[dr]^{\pi_2} & \\
\PP^n&  &U.
}
\]
And through each point in $U$ passes a closed subscheme $V\subseteq U$ such that $\pi^*\F|_{\PP^{q+1}\times V}$ is an isotrivial unfolding of a foliation by curves $\F_0$ in $\PP^{q+1}$ parametrized by $V$.
Indeed, we can get $\F_0$ by taking a plane $P$ representing a point in $V$ and so we have $\F_0=p(\F|_P)$.
We can further restrict $V$ to its reduced structure $V^{\mathrm{red}}$ and then even more to the non-empty open set of regular points of $V^{\mathrm{red}}$ in order to be able to apply \cref{proptrans}.
By \cref{proptrans}, there is an open subset $Y$ of $\PP^{q+1}$ and an open subset $W$ of $V^{\mathrm{red}}$ such that the restriction of $\pi^*\F|_{\PP^{q+1}\times V}$ to $Y\times W$ is a \emph{transversal and isotrivial} unfolding of $\F_0|_Y$.
Now, we have two possibilities, either $\F_0$ have rational infinitesimal symmetries or not.

\par If $\F_0$ have no rational symmetries, then the sheaf $\mathfrak{aut}(\F_0)$ is trivial. 
Then, because of the short exact sequence
\[ 0\to T\F_0\to \mathfrak{aut}(\F_0)\to \mathfrak{u}(\F_0)\to 0, \]
 also $\mathfrak{u}(\F_0)=0$.
As was said before, by \cref{proptrans} we can restrict the unfolding $\pi^*\F|_{\PP^{q+1}\times V}$ to $Y\times W$ in such a way that the unfolding is now transversal and isotrivial on $Y\times W$.
Note that  $\mathfrak{u}(\F_0|_Y)= \mathfrak{u}(\F_0)|_Y=0$. 
So the unfolding is trivial and thus $\pi^*\F|_{\PP^{q+1}\times V}$ is birationally equivalent to the pull-back of a foliation by curves on $\PP^{q+1}$.

Notice also that foliations with no rational symmetries form a dense open set on  $ \PP H^0(\PP^{q+1}, T\PP^{q+1}(d-1))$, so either a dense open set of $P\in U$ verify that $\F|_P$ have no rational symmetries or, on the contrary, every $P\in U$ is such that $\F|_P$ have a rational symmetry.

Summarizing we have proved the following.
\begin{prop}\label{casiahi}
Let $\F$ be a degree $d$ foliation in $\PP^n$ of codimension $q$.
Suppose the condition 
\[ (n-q-1)(q+1)> (q+2)\binom{d+q+1}{d}-\binom{d+q}{d-1} \]
is satisfied.
Then there is an open set $U\subset \mathrm{Gr}^{n}_{q+1}$ and a rational morphism 
\[
 \PP^{q+1}\times U\stackrel{\pi}{\to} \PP^n
\]
such that the following alternative holds.
Either one have that for every $q+1$-linear subspace $P\in U$, the restriction $\F|_P$ have rational symmetries; or there is, for each $P$ in a dense open subset of $U$, a subvariety $V_P\subseteq  U$ of codimension at most $ (q+2)\binom{d+q+1}{d}-\binom{d+q}{d-1}$ such that $\pi^*\F|_{\PP^{q+1}\times V_P}$ is the pull-back of a foliation by curves on $\PP^{q+1}$. 
\end{prop}
In codimension $1$ we can improve this result.

\begin{teo}\label{teopapa}
Let  $\F$ be a foliation of codimension $1$ and degree $d$ on $\PP^n(\C)$.
Suppose $\dim(\sing(\F))\leq n-2$ and the condition
\[
 2(n-2)> 3\binom{d+2}{d}-\binom{d+1}{d-1}
\]
is satisfied.
Then we have the following alternative:
\begin{enumerate}
\item Either there exist  an open subset $U$ of $\mathrm{Gr}_2^{q+1}$, and a map $\pi:U\times\PP^{2}\to \PP^n$ such that for each $P\in U$ there is a subvariety $V_P\subseteq U$ containing $P$ and of codimension at most $3\binom{d+2}{d}-\binom{d+1}{d-1}$ such that $\pi^*\F|_{\PP^2\times V_P}$ is pull-back of a foliation by curves on $\PP^{2}$. 
\item  Or there are holomorphic varieties $\~{S}$ and $\~{Y}$ and a meromorphic map with discrete (not necessarily finite) generic fiber, $\phi:\~{S}\times\PP^{2}\to \PP^n$,  such that $\phi^*\F$ has a meromorphic first integral.
\end{enumerate}
\end{teo}
\begin{proof}
The inequality in the hypotheses allow us to apply \cref{casiahi}.
Then, either we have the map $\pi$, and subvarieties $V_P$ as in \cref{casiahi} or we have that for every $2$-dimensional linear subspace $P\subset \PP^n$ the restriction $\F|_P$ is a foliation with rational symmetries.

\medskip
In the second case $\F|_P$, has a rational integrating factor $f$.
Then, by \cite{cerveau-mattei}, if we take $p:\~{Y}\to P\setminus \mathrm{Div}(f)$ to be the universal covering map, $p^*\F|_P$ has 
a first integral.
\par Then $(p\times \mathrm{id})^* \phi^* \F$ is a transveral isotrivial unfolding of a foliation over $\~{Y}$ parametrized by $S$.
It is an unfolding of a foliation with a first integral. 
As  $(p\times \mathrm{id})^* \phi^* \F$ is the unfolding of a foliation with a first integral defined on a simply connected space, then by \cite[5.3]{suwa}, $(p\times \mathrm{id})^* \phi^* \F$ has itself a first integral (Suwa's original result implies the existence of a local first integral, which we can extend globally on account of $Y$ being simply connected).
\end{proof}
We can express this result more succinctly as follows.

\begin{cor}\label{coropapa}
Let $\F$ be a foliation in $\PP^n(\C)$ satisfying the hypotheses of \cref{teopapa}.
Then either $\F$ is given by a closed rational form or a generic leaf of $\F$ contains algebraic varieties of codimension at most $3\binom{d+2}{d}-\binom{d+1}{d-1}$.
\end{cor}


\begin{bibdiv}
\begin{biblist}

\bib{structural}{article}{
  title={A structural theorem for codimension one foliations on $\PP^n$ , $n\geq 3$, with application to degree three foliations},
  author={Cerveau, Dominique}
  author={Lins Neto, Alcides},
  journal={Ann. Sc. Norm. Super. Pisa Cl. Sci.},
  volume={12},
  number={1},
  pages={1--41},
  year={2013}
}

\bib{god-vey}{article}{
  title={Complex codimension one singular foliations and Godbillon-Vey sequences},
  author={Cerveau, Dominique},
  author={Lins Neto, Alcides},
  author={Loray, Frank},
  author={Pereira, Jorge Vitorio},
  author={Touzet, Fr{\'e}d{\'e}ric},
  journal={Moscow mathematical journal},
  volume={7},
  pages={21--54},
  year={2007}
}

\bib{cerveau-mattei}{book}{
  title={Formes int{\'e}grables holomorphes singuli{\`e}res},
  author={Cerveau, Dominique}
  author={Matt{\'e}i, J.F.},
  year={1982},
  publisher={Soci{\'e}t{\'e} Math{\'e}matique de France}
}



\bib{transafin}{article}{
  title={Transversely affine foliations on projective manifolds},
  author={Cousin, Ga{\"e}l },
  author={Pereira, Jorge Vit{\'o}rio},
  journal={arXiv preprint arXiv:1305.2175},
  year={2013}
}

\bib{LPT}{article}{
 title={Singular foliations with trivial canonical class},
  author={Loray, Frank},
  author={Pereira, Jorge Vitorio},
  author={Touzet, Fr{\'e}d{\'e}ric},
  journal={arXiv preprint arXiv:1107.1538},
  year={2011}
}

\bib{jvp}{article}{
  title={Transformation groups of holomorphic foliations},
  author={Pereira, Jorge Vit{\'o}rio},
  author = {S{\'a}nchez, Percy Fern{\'a}ndez},
  journal={Comm. Anal. Geom},
  volume={10},
  number={5},
  pages={1115--1123},
  year={2002}
}

\bib{paper1}{article}{
title={Families of distributions and Pfaff systems under duality},
author={Quallbrunn, F.},
journal={arXiv:1305.3817}
}

\bib{suwa81}{article}{
  title={A theorem of versality for unfoldings of complex analytic foliation singularities},
  author={Suwa, Tatsuo},
  journal={Inventiones mathematicae},
  volume={65},
  number={1},
  pages={29--48},
  year={1981},
  publisher={Springer}
}

\bib{suwa83}{article}{
  title={Unfoldings of complex analytic foliations with singularities},
  author={Suwa, Tatsuo},
  journal={Japanese journal of mathematics. New series},
  volume={9},
  number={1},
  pages={181--206},
  year={1983},
  }

\bib{Suwasing}{article}{
title={Structure of the singular set of a complex analytic foliation},
  author={Suwa, Tatsuo},
  journal={Preprint series in mathematics. Hokkaido University},
  year={1988},
  publisher={Department of Mathematics, Univ. Hokkaido},
 volume={33}
}
\bib{suwa}{article}{
  title={Unfoldings of codimension one complex analytic foliation singularities},
  author={Suwa, Tatsuo},
  year={1992}
}

\end{biblist}
\end{bibdiv}
\end{document}